\pdfoutput=1 % for arxiv
\documentclass{sl_style/sl}     %% This the CLASS for Studia Logica,
\usepackage{sl_style/slsec}     %% PACKAGES for Studia Logica
\usepackage{sl_style/stud_log} 
\usepackage{sl_style/slfoot}
\usepackage{sl_style/slthm}     %% Theorem-like definitions as in amsthm,
\usepackage[T1]{fontenc}
\usepackage{lmodern}
\usepackage{microtype}
\usepackage{enumerate}
\usepackage{booktabs}
\usepackage{amsmath}
\usepackage{amssymb}
\usepackage{mathrsfs}

\theoremstyle{definition}
\newtheorem{theorem}{Theorem}[section]

\newtheorem{lemma}[theorem]{Lemma}
\newtheorem{proposition}[theorem]{Proposition}

\theoremstyle{definition}
\newtheorem{remark}[theorem]{Remark}
\newtheorem{example}[theorem]{Example}
\newtheorem{definition}[theorem]{Definition}
\newtheorem{convention}[theorem]{Convention}

\newcommand{\0}{\varnothing}
\def\ie{{i.e.}}
\def\eg{{e.g.}}

\newcommand{\mop}{\mathop{\Box}\nolimits}
\newcommand{\mord}{\mathord{\Box}}
\newcommand{\power}{\mathscr{P}}
\newcommand{\zenbu}[1]{\mathop{\forall#1}}
\newcommand{\aru}[1]{\mathop{\exists#1}}
\newcommand{\ess}{\mathrm{e}}
\newcommand{\Th}{\mathrm{Th}}
\newcommand{\Def}{\mathrm{Def}}
\newcommand{\uf}{\mathrm{Uf}}
\newcommand{\Uf}{\uf}
\DeclareMathOperator{\ue}{ue}
\newcommand{\tp}{\mathrm{tp}}
\newcommand{\closed}[1]{E_{#1}}
\HeadingsInfo{Yamamoto}{Correspondence, Canonicity, and Model Theory for Monotonic Modal Logics}

\begin{document}

\setcounter{page}{1}     %%  Initial page number. Do not change. 

%%%%%%  FORMAT OF AUTHOR AND TITLE INFORMATION:
 
%%%%%%%%%%%%%%%%%%% The (default) case of one author
%%%%%%%%%%%%%%%%%%% and one or two lines of title

\AuthorTitle{Kentar\^o Yamamoto}{Correspondence, Canonicity, and \newline Model Theory for Monotonic Modal Logics}

%%%%%%  INFORMATION FOR FOOTER OF FIRST PAGE
%%%%%%  will be inserted by the Editorial Office.
   
\PresentedReceived{Name of Editor}{December 1, 2005}

\begin{abstract}
  We investigate the role of \emph{coalgebraic predicate logic},
  a logic for neighborhood frames first proposed by Chang,
  in the study of monotonic modal logics.
  We prove analogues of the Goldblatt-Thomason Theorem 
  and Fine's Canonicity Theorem
  for classes of monotonic neighborhood frames
  closed under elementary equivalence in coalgebraic predicate logic.
  The elementary equivalence here can be relativized
  to the classes of monotonic, quasi-filter, augmented quasi-filter, filter, or augmented filter
  neighborhood frames, respectively.
  The original, Kripke-semantic versions of the theorems follow
  as a special case concerning
  the classes of augmented filter neighborhood frames.
\end{abstract}


\section{Introduction}
Monotonic modal logics generalize normal modal logics
by dropping the K axiom $\mord(p\to q)\to (\mord p\to\Box q)$ 
and instead requiring only that $\vdash \phi \to \psi$ imply $\vdash \mord \phi \to \mord \psi$.
There are a number of reasons for relaxing the axioms of normal modal logics
and considering monotonic modal logics.
For instance, monotonic modal logics are considered more appropriate
to describe the ability of agents or systems
to make certain propositions true in the context of games and open systems~%
\cite{Parikh85thelogic,pauly01:_logic_social_softw,Alur97alternating-timetemporal}.
The standard semantics for monotonic modal logics is provided
by monotonic neighborhood frames (see, e.g., \cite{hansen03:_monot_modal_logic}).

Just as the first-order language with a relation symbol
is a useful correspondence language for Kripke frames,
it is natural to consider what would be a useful correspondence language 
for monotonic neighborhood frames.
Litak et al.~\cite{2017arXiv170103773L} studied \emph{coalgebraic predicate logic} (CPL)
as a logic that plays that role
and proved a characterization theorem in the style of van Benthem and Rosen~\cite{doi:10.1093/logcom/exv043}.
In this article, we continue that path for monotonic neighborhood frames
and prove variants 
 of the Goldblatt-Thomason theorem~\cite{10.1007/BFb0062855} and the Fine canonicity theorem~\cite{article} 
in the setting of coalgebraic predicate logic.

We will deal with a relativized notion of CPL-elementarity,
relativized to subclasses of the class of monotonic neighborhood frames.
There are several important subclasses to consider:
the class of filter neighborhood frames,
providing a more general  semantics~\cite{10.2307/20014778,gerson76:_neigh_frame_t_no_equiv_relat_frame} for normal modal logics than relational semantics;
the class of quasi-filter neighborhood frames,
providing a semantics for regular modal logics;
the class of augmented quasi-filter neighborhood frames,
providing a less general semantics for regular modal logics;
and the class of augmented filter neighborhood frames,
which are Kripke frames in disguise
\cite{chellas80:_modal_logic,pacuit_neighborhood_2017}.

\begin{table}[htbp]
  \centering
  \small
  \begin{tabular}{lp{.65\textwidth}}  \toprule
    Subclass & Closed under~... \\ \midrule
    monotonic &   supersets \\
    quasi-filter & supersets, intersections of nonempty finite families of neighborhoods \\
    augmented quasi-filter &  supersets, intersections of nonempty families of neighborhoods \\
    filter &   supersets,  intersections of finite families of neighborhoods \\
    augmented filter &   supersets, intersections of families of neighborhoods \\ \bottomrule
  \end{tabular}
  \caption{Classes of monotonic neighborhood frames and their definitions}
  \label{tab:defs}
\end{table}

The analogue of the Goldblatt-Thomason theorem in this article
is that a class of monotonic neighborhood frames
closed under CPL-elementarity relative to
any of the classes of neighborhood frames in Table~\ref{tab:defs}
is modally definable if and only if 
it is closed under disjoint unions, bounded morphic images,
and generated subframes,
and it reflects ultrafilter extensions;
and the analogue of Fine's theorem we will prove
states that a sufficient condition for the canonicity
of a monotonic modal  logic is that
it is complete with respect to the class of monotonic neighborhood frames it defines
and that that class is closed under CPL-elementarity relative to
any of the classes of neighborhood frames in Table~\ref{tab:defs}.

% Why is the correspondence result here?
The relevance of coalgebraic predicate logic
in this article is
that many monotonic modal logics
define classes of monotonic neighborhood frames
that are CPL-elementary.
For instance, the monotonic modal logics
axiomatized by  formulas of the form
\begin{equation}
  \label{eq:veryfirst}
  \langle \text{purely propositional positive formula} \rangle
  \rightarrow \langle \text{positive formula} \rangle
\end{equation}
are determined by CPL-elementary classes of monotonic neighborhood frames
(see Remark~\ref{rem:important}).
In addition, relative to the class of augmented quasi-filter frames,
all monotonic modal logics axiomatized by Sahlqvist formulas
are CPL-elementarily determined~%
%\cite{2016arXiv160308202P}
(see Example~\ref{exa:T}). % ?????????
Further discussion regarding the relevance of this language
in the context of Fine's theorem is in Remark~\ref{rem:excuse}.

Since augmented filter frames are Kripke frames in disguise
(see also Example~\ref{exa:T}),
our result regarding classes elementary relative to the class of
augmented filter frames generalizes
the original, Kripke-semantic Goldblatt-Thomason theorem.
Also,
our Goldblatt-Thomason theorem concerns elementary classes
like the original theorem,
whereas some existing Goldblatt-Thomason theorems
such as \cite{lmcs:1167} or \cite{kurz2007goldblatt}
deal with classes closed under ultrafilter extensions.
%This makes the latter theorems rather immediate consequences of the Birkoff variety theorem.

The article is organized as follows.
In \S~\ref{sec:preliminaries},
we recall standard concepts
in the semantics of monotonic modal logic
and introduce the language for neighborhood frames.
In \S~\ref{sec:model-theory-neighb},
we give an overview of the model theory of neighborhood frames
for this language.
We also define a two-sorted first-order language
(Definition~\ref{defi:twosorted})
and a translation of coalgebraic predicate logic into it
(Proposition~\ref{prop:translation}),
which are used later to explain the existence of $\aleph_0$-saturated models
of languages of coalgebraic predicate logic
(Proposition~\ref{prop:saturation}).
In \S~\ref{sec:canonicity},
we prove the main lemmas of this articles.
In \S~\ref{sec:appl-main-lemm},
we give the applications of the main lemmas,
which are analogues of the Goldblatt-Thomason Theorem
and Fine's Canonicity Theorem.

The presentation of the results in this article
does not presuppose the reader's prior knowledge of coalgebras
or coalgebraic predicate logic.

\section{Preliminaries}\label{sec:preliminaries}
\subsection{Languages and structures}
In this subsection,
we recall standard definitions in neighborhood semantics of modal logic
and the language coalgebraic predicate logic introduced in \cite{Chang1973-CHAMMT} and \cite{2017arXiv170103773L}
to describe neighborhood frames.

We define languages of coalgebraic predicate logic
relative to sets of nonlogical symbols here;
this is so that we can use expansions of the smallest language
in proofs in \S~\ref{sec:canonicity}.

\begin{definition}\label{def:languages-structures}
  \hfill
  \begin{enumerate}[(i)]
  \item\label{item:CPL}
    Let $\sigma$ be the set of atomic formulas of 
    some language of first-order logic.
    The \emph{language of coalgebraic predicate logic} $L$ \emph{based on} $\sigma$
    is the least set of formulas
    containing $\sigma$ and
    closed under Boolean combinations, existential quantification,
    and formation of formulas of the form $x\mop\lceil y \colon \phi
    \rceil$,
    where $\phi \in L$, and $x$ and $y$ are variables.
    To save space, we sometimes write $x\mop_y\phi$ or even
    $x\mop\phi$
    for $x\mop\lceil y : \phi \rceil$.
    For a language $L_0$ of first-order logic,
    the \emph{language of coalgebraic predicate logic based on} $L_0$
    is defined to be the language of coalgebraic predicate logic based on
    the set of atomic formulas of $L_0$.
    We write $L_=$ for the language
    of coalgebraic predicate logic based on the empty language,
    \ie, the language with just the equality symbol.
  \item
    Let $L_0$ be a language of first-order logic
    and $L$ the language of coalgebraic predicate logic based on $L_0$.
    An \emph{$L$-structure} $F = (F, N^F)$ is an $L_0$-structure $F$
    with an additional datum $N^F : F \to \power(\power(F))$,
    where $\power$ is the powerset operation. 
    The map $N^F$ is called the \emph{neighborhood function} of $F$.
    A set $U \in N^F(w)$ is called a \emph{neighborhood} of $w$.
    If $L_0$ is the empty first-order language,
    the $L$-structures are exactly the \emph{neighborhood frames}.
  \item
    A neighborhood frame $F$ is \emph{monotonic}
    if for every $w \in F$ the family $N^F(w)$ is closed under supersets.
    $F$ is a \emph{quasi-filter} neighborhood frame if for every $w \in F$
    the family $N^F(w)$ is closed under intersections of nonempty finite families of neighborhoods.
    $F$ is a \emph{filter} neighborhood frame if
    it is a quasi-filter frame and
    for every $w \in F$ the family $N^F(w)$ is nonempty.
    $F$ is an \emph{augmented quasi-filter} neighborhood frame if for every $w \in F$
    the family $N^F(w)$ is either empty or a \emph{principal upset}
    in the Boolean algebra $\power(F)$,
    \ie, there exists $U_0 \subseteq F$ such that
    $U \in N^F(w) \iff U_0 \subseteq U$.
    Finally, $F$ is an \emph{augmented filter} neighborhood frame
    if for every $w \in F$ the family $N^F(w)$ is a principal upset.
  \end{enumerate}
\end{definition}

One might object to  calling the language in Definition~\ref{def:languages-structures}(\ref{item:CPL}) the language of ``coalgebraic predicate logic''
since this is essentially what Chang introduced in \cite{Chang1973-CHAMMT}
whereas the language in Litak et al.~\cite{2017arXiv170103773L}
is applicable to general coalgebras.
We use the name ``coalgebraic predicate logic'' in this article
because Chang's language does not have a name,
and technically speaking we do not use Chang's syntax,
which imposes a more strict rule
regarding variables bound by modal operators.

\begin{example}\label{exa:deftopo}
  For a topological space $X = (X, \tau)$,
  we associate a neighborhood frame $X^* = (X, N)$ defined by
  \[
    U \in N(w) \iff w \in U^\circ,
  \]
  where ${}^\circ$ denotes topological interior.
  We call a monotonic neighborhood frame of the form $X^*$ 
  a \emph{topological neighborhood frame}.
  Recall the satisfaction predicate $\Vdash_{\mathrm{top}}$ for
  topological semantics
  and
  the satisfaction predicate $\Vdash_{\mathrm{nbhd}}$ for neighborhood semantics
  (see, e.g., \cite{chellas80:_modal_logic,vanBenthem2007} for more details):
  \[
    M, w \Vdash_{\mathrm{top}} \mord\phi \iff
    w \in \{ w' \mid M, w' \Vdash_{\mathrm{top}}\phi \}^\circ
  \]
  and
  \[
    M', w \Vdash_{\mathrm{nbhd}} \mord\phi \iff
    \{ w' \mid M, w' \Vdash_{\mathrm{nbhd}}\phi \} \in N(w),
  \]
  where $M$ is a topological model,
  $M'$ is a neighborhood model,
  and $N$ is the neighborhood function of the neighborhood frame of $M'$.
  It is then easy to see that for every $w \in X$,
  every modal formula $\phi$,
  every topological model $M$ based on $X$, and
  every neighborhood model $N$ based on $X^*$,
  if the valuations of $M$ and $N$ are the same, then
  \[
    M, w \Vdash_{\mathrm{top}} \phi \iff N,w \Vdash_{\mathrm{nbhd}} \phi.
  \]
\end{example}

\begin{definition}
  Let $L$ be a language of  coalgebraic predicate logic and $F$  an
  $L$-structure.
  We define the satisfaction predicate $F \models \phi$
  for a sentence $\phi \in  L$.
  It is convenient to define the predicate for the expanded language $L(F)$
  of coalgebraic predicate logic.
  In general, for $A \subseteq F$,
  we define $L(A)$ to be the language of coalgebraic predicate logic
  that has all symbols of $L$
  and for each $w \in A$ a constant symbol $w$
  that is intended to be interpreted as $w$ itself.
  Now, $F$ is an $L(F)$-structure in the obvious way.
  We define the satisfaction predicate $F \models \phi$
  for $\phi \in L(F)$.
  The predicate is defined by recursion on $\phi$.
  For symbols of first-order logic in $L$,
  the predicate is defined in the usual way.
  For $\phi = w\mop_y\phi_0$, we define
  \[
    F \models w \mop_y\phi_0(y) \iff
    \phi_0(F) \in N^F(w)
  \]
  where
  \[
    \phi_0(F) = \{v \in F \mid F \models \phi_0(v)\}
  \]
  and $\phi_0(v)$ stands for the substitution instance of $\phi_0(y)$
  with $v$ substituted for $y$.
\end{definition}

The use of constant symbols interpreted as themselves is standard practice in model theory (see, e.g., \cite{marker02:_model_theor});
it makes the notation and definitions much simpler, particularly in later parts of this article where we deal with types with parameters.

\begin{example}\label{exa:B}
  Consider the B axiom $p \rightarrow \mord \neg \mord \neg p$.
  We see that this modal formula has a local frame correspondent
  relative to the class of monotonic neighborhood frames
  in the language $L_=$.
  Consider the validity of the B axiom
  for a monotonic neighborhood frame $F$ and $w \in F$.
  By the monotonicity of $F$,
  the usual minimum valuation argument (see, e.g., \cite{Blackburn:2001:ML:381193}) applies:
  the B axiom is valid at $w$ if and only if
  its consequent is true
  under the minimum valuation that makes its antecedent true,
  which is the valuation that sends $p$ to the set $\{w\}$.
  The latter condition is expressible by a formula in $L_=$:
  \[
    w \mop_y (\neg y \mop_z z \neq w).
  \]
\end{example}
\begin{remark}\label{rem:important}
  It can be shown likewise that modal formulas of the form~(\ref{eq:veryfirst})
  have frame correspondents relative
  to the class of monotonic neighborhood frames.
  A formula of the form~(\ref{eq:veryfirst})
  is  what is called a KW formula in \cite{hansen03:_monot_modal_logic}
  and axiomatizes a monotonic modal logic complete with respect to
  the class of monotonic neighborhood frames that it defines.
  Hence, 
  the monotonic modal logics
  axiomatized by such formulas
  are determined by CPL-elementary classes (see Definition~\ref{defi:CPL-elementary}) of monotonic neighborhood frames.
\end{remark}
\begin{example}\label{exa:T}
  Consider the 4 axiom $\mord p \rightarrow \mord\mord p$.
  We show that this modal formula has a local frame correspondent
  relative to the class of augmented quasi-filter neighborhood frames
  in the same language $L_=$ as above.
  Consider the validity of the 4 axiom for an augmented quasi-filter neighborhood frame $F$
  and $w \in F$.
  If $w \in F$ is \emph{impossible},
  \ie, $N^F(w) = \0$,
  then the 4 axiom is valid at $w$.
  Note that  by monotonicity
  $w$ is impossible if and only if $F \not \in N^F(w)$,
  \ie, $F \models \neg w \mop_y y = y$.
  Otherwise,
  we can again use the minimum valuation argument.
  Here,
  the minimum interpretation of $p$ that makes the antecedent true
  is $R[w]$ because $F$ is an augmented quasi-filter neighborhood frame,
  where $R \subseteq F \times F$ is the binary relation defined by
  \begin{equation}
    xRy \iff \{z \in F \mid z \neq y\} \not \in N^F(x)
    \quad(\mathord{\iff} F \models \neg x\mop_z z \neq y).\label{eq:binary}
  \end{equation}
  To summarize, the 4 axiom has the local frame correspondent
  \[
    \neg w \mop \lceil y\colon y = y \rceil \mathrel{\vee}
    (w \mop \lceil y\colon y = y \rceil \mathrel{\wedge}
    w \mop \lceil y_1 \colon
      y_1 \mop \lceil y_2 \colon
       \neg y_2 \mop \lceil z \colon z \neq w\rceil
      \rceil
    \rceil).
    \]
  In fact,
  since the accessibility relation $R$ and the set of impossible worlds
  are definable in $L_=$ as we have seen above,
  the first-order frame correspondence language in \cite{2016arXiv160308202P}
  translates into $L_=$,
  and thus all Sahlqvist formulas have frame correspondents in $L_=$
  relative to the class of augmented quasi-filter neighborhood frames.

  The displayed formula~(\ref{eq:binary}) can be used to define
  the class of augmented quasi-filter neighborhood frames by
  coalgebraic predicate logic as well.
  Write $R[x]$ for the set of $y \in F$ satisfying (\ref{eq:binary})
  for an arbitrary monotonic neighborhood frame $F$ and $x \in F$.
  We see that a monotonic neighborhood frame $F$ is augmented quasi-filter
  if and only if either $N^F(w)$ is impossible, or $R[w] \in N^F(w)$
  for every $w \in F$, \ie, $F$ satisfies the $L_=$-sentence
  \[
    \zenbu{x} [(x \mop_y y = y)
    \mathrel{\to} x \mop_y \neg (x \mop_z z \neq y)].
  \]
  Indeed, we have seen the ``only if'' direction in the last paragraph;
  to see the ``if'' direction,
  observe that $R[w] = \bigcap N^F(w)$.
\end{example}

\begin{example}\label{exa:topo}
  Recall that for a topological space $X$
  the \emph{specialization preorder} of $X$
  is the preorder $\lesssim$ on $X$ defined by
  \[
    x \lesssim y \iff  x \in \overline{\{y\}},
  \]
  where $\overline{(\cdot)}$ denotes topological closure.
  A space $X$ is T${}_0$ if and only if $\lesssim$ is a partial order,
  and $X$ is T${}_1$ if and only if $\lesssim$ is a discrete partial order.
  Note that the specialization preorder of a topological space $X$
  is ``definable'' in coalgebraic predicate logic in the sense that
  \begin{equation}
    x \lesssim y \iff X^* \models \neg x\mop_zz \neq y.\label{eq:annoying}
  \end{equation}
  Hence, the images under ${}^*$ of the classes of T${}_0$ and T${}_1$ spaces
  are CPL-elementary relative to the class of topological neighborhood frames:
  $X$ is T${}_0$ if and only if 
  $X^* \models \zenbu{z}\zenbu{w} %
  (z \lesssim w \land w \lesssim z \rightarrow w = z)$,
  and 
  $X$ is T${}_1$ if and only if 
  $X^* \models \zenbu{z}\zenbu{w} %
  (z \lesssim w \rightarrow w = z)$,
  where $x \lesssim y$ abbreviates the formula of coalgebraic predicate logic
  on the right-hand side of the displayed formula~(\ref{eq:annoying}).
\end{example}

\begin{definition}
  Let $F$ and $F'$ be neighborhood frames.
  A function $f : F \to F'$ is a \emph{bounded morphism}
  if for each $w \in F$:
  \[f^{-1}(U') \in N^{F}(w) \implies U' \in N^{F'}(f(w)) \tag{``forth''}\]
  and
  \[U' \in N^{F'}(f(w)) \implies f^{-1}(U') \in N^{F}(w) \tag{``back''}.\]
\end{definition}

\begin{lemma}[\cite{Dosen1989}]
  Let $F$ and $F'$ be monotonic neighborhood frames
  and $f : F \to F'$ be a function that satisfies the ``forth'' condition.
  Suppose in addition that for all $U' \in N^{F'}(f(w))$
  there exists $U \in N^{F}(w)$ such that $f(U) \subseteq U'$.
  Then $f$ is a bounded morphism.
\end{lemma}
\begin{proof}
  By assumption, if $U' \in N^{F'}(w)$,
  then there exists $U$ such that $f^{-1}(U') \supseteq U \in N^G(w)$;
  by monotonicity, we have $f^{-1}(U') \in N^G(w)$.
\end{proof}
Note that bounded morphisms between monotonic neighborhood frames
clearly satisfy the assumption of this lemma.

\subsection{Algebraic concepts}
In this subsection,
we recall some standard definitions from the algebraic treatment of modal logic;
for the standard notions
 that we do not define here, see \cite{venema07:_algeb_coalg}.

 First, we recall basic definitions regarding the algebraic treatment
 of monotonic modal logic.
\begin{definition}
  A \emph{monotonic Boolean algebra expansion}
  (\emph{BAM} for short) $A = (A, \mord^A)$
  is a Boolean algebra $A$ with an additional datum
  $\mord^A : A \to A$, a function that is \emph{monotonic},
  \ie, for all $a, b \in A$ we have
  $a \le b \implies \mord^A (a) \le \mord^A(b)$.
  
\end{definition}
% \begin{remark}
%   Diamond vs Box.
% \end{remark}

\begin{lemma}
  Let $F$ be a monotonic neighborhood frame.
  The function $\mord^F : \power(F) \to \power(F)$ defined by
  \[
    X \mapsto \{w \in F \mid X \in N^F(w)\}
  \]
  is monotonic.\qed
\end{lemma}

\begin{definition}[\cite{Dosen1989}]
  The \emph{complex algebra} $F^+$ of a monotonic neighborhood frame $F$
  is the BAM $(\power(F), \mord^F)$,
  where $\power(F)$ is the Boolean algebra of the powerset of $F$.
\end{definition}

\begin{proposition}
  Let $F$ and $F'$ be monotonic neighborhood frames.
  A function $f: F \to F'$ is a bounded morphism
  if and only if $f^+ : F'^+ \to F^+$ defined by $f^+(X) = f^{-1}(X)$
  is a homomorphism.\qed
\end{proposition}

 Since this article concerns canonicity,
 we need to recall definitions regarding canonical extensions.
\begin{definition}
    Let $B$ be a Boolean algebra.
    The \emph{canonical extension} $B^\sigma$ of $B$
    is the Boolean algebra of the powerset
    of the set $\uf(B)$ of ultrafilters in $B$.
    An element of $B^\sigma$ of the form $[a] := \{u \in \uf(B)\mid a \in u\}$
    for a fixed $a \in B$ is called \emph{clopen}.
    Meets and joins of clopen elements of $B^\sigma$
    are \emph{closed} and \emph{open}, respectively.
    The sets of closed and open elements of $B^\sigma$
    are denoted $K(B^\sigma)$ and $O(B^\sigma)$, respectively.
\end{definition}

\begin{proposition}
  For a Boolean algebra $B$, the map $[-] : B \to B^\sigma$ is an embedding.
\end{proposition}
\begin{proof}
  See, \eg, \cite{venema07:_algeb_coalg}.
\end{proof}

\begin{definition}[see, e.g., \cite{venema07:_algeb_coalg}]
  \mbox{}
  \begin{enumerate}[(i)]
  \item 
    Let $A = (A, \mord)$ be a BAM.
    The \emph{canonical extension} $A^\sigma = (A^\sigma, \mord^\sigma)$
    of $A$
    is the canonical extension of the Boolean algebra $A$
    expanded by the function $\mord^\sigma$,
    where
    \[
      \mord^\sigma(u)
      = \bigvee_{u \supseteq x \in K(A^\sigma)}
      \bigwedge_{x \subseteq a \in A} \mord(a).
    \]
  \item
    A set $\Delta$ of modal formulas is \emph{canonical}
    if for every BAM $A \models \Delta$ we have $A^\sigma \models \Delta$.
  \end{enumerate}
\end{definition}

\begin{proposition}
  For a BAM $A = (A, \mord)$,
  the function $\mord^\sigma$ is monotonic,
  and thus the canonical extension $A^\sigma = (A^\sigma, \mord^\sigma)$
  is again a BAM.
\end{proposition}
\begin{proof}
  See, \eg, \cite{venema07:_algeb_coalg}.
\end{proof}

\begin{remark}
  Canonical extensions can be defined for larger classes of algebras.
  We stick to BAMs in this article
  since they admit the most natural definition for $\mord^\sigma$,
  among other reasons.
\end{remark}

\begin{definition}[\cite{hansen03:_monot_modal_logic}]
  \hfill\label{def:algebraic-concepts}
  \begin{enumerate}[(i)]
  \item\label{item:please}
    Let $A$ be a BAM.
    The \emph{ultrafilter frame} of $A$
    is a neighborhood frame $(\uf (A), N^\sigma)$
    with $N^\sigma$ defined by
    \begin{equation}
      \label{eq:ultraframe}
      U \in N^\sigma (u)
      \iff \aru{K \subseteq U} \zenbu{a \in A}
      ([a] \supseteq K \Rightarrow
      \mord(a) \in u),
    \end{equation}
    where
    $u \in \uf (A)$, and
    $K$ ranges over closed elements of
    $A^\sigma = \power(\uf A)$.
    We denote the ultrafilter frame of $A$ by $\uf(A)$.
  \item
    Let $F$ be a monotonic neighborhood frame.
    The \emph{ultrafilter extension} $\ue F$ of $F$ is $\uf(F^+)$.
  \end{enumerate}
\end{definition}

\begin{proposition}
  Let $A$ be a BAM.
  \begin{enumerate}[(i)]
  \item $\uf (A)$ is monotonic.
  \item $(\uf (A))^+ = A^\sigma$. \qed
  \end{enumerate}
\end{proposition}

Finally, we define a few notions necessary to state our
Goldblatt-Thomason theorem.
\begin{definition}[\cite{hansen03:_monot_modal_logic}]
  For a disjoint family $((F_i, N^i) \mid i \in I)$ of monotonic neighborhood frames,
  the \emph{disjoint union} of the family is $(F, N)$,
  where $F = \bigsqcup_i F_i$
  and $N$ is a neighborhood function defined by 
  $U \in N(w) \iff U \cap F_i \in N^i(w)$.
  A monotonic neighborhood frame $F$ is a \emph{bounded morphic image} of another
  $F'$ if there is a surjective bounded morphism $F' \twoheadrightarrow F$.
  A monotonic neighborhood frame $F$ is a \emph{generated subframe} of another $F'$
  if $F \subseteq F'$,
  and the inclusion map $F \hookrightarrow F'$ is a bounded morphism.
\end{definition}
% Disjoint unions, bounded morphic images, and generated subframes
% are the dual notions
% of direct products, subalgebras, and homomorphic images
% under the dual equivalence of monotonic neighborhood frames and BAMs
% \cite{hansen03:_monot_modal_logic}.

\section{Model theory of neighborhood frames}\label{sec:model-theory-neighb}
In this section,
we recall as well as develop results in the model theory of neighborhood frames
and coalgebraic predicate logic. 
\subsection{Standard concepts in first-order model theory}
Here, we define concepts that have counterparts
in classical first-order model theory.
\begin{definition}
  Let $L$ be a language of coalgebraic predicate logic,
  $F$  an $L$-structure,
  and $A \subseteq F$.
  A subset $X \subseteq F$ is \emph{$A$-definable in $F$}
  if there is an $L$-formula $\phi(x; \bar y)$
  and a tuple $\bar a$ of elements of $A$ (notation: $\bar a \in A$)
  such that $X = \phi(F; \bar a)$.
  A subset $X$ is \emph{definable in $F$} if it is $F$-definable in $F$.
\end{definition}

\begin{definition}
  \hfill
  \begin{enumerate}[(i)]
  \item
    A set of $L$-sentences is called an $L$-\emph{theory}.
  \item 
    Let $L$ be a language of coalgebraic predicate logic
    and $F$ be an $L$-structure.
    The \emph{full $L$-theory} $\Th_L(F)$ of $F$
    is the set of $L$-sentences $\phi$ such that $F \models \phi$.
  \item
    Two $L$-structures $F, F'$ are \emph{$L$-elementarily equivalent},
    or $F \equiv_L F'$, if $\Th_L(F) = \Th_L(F')$.
  \end{enumerate}
\end{definition}
For the rest of this section,
we fix a language $L$ of coalgebraic predicate logic
and a monotonic $L$-structure $F$.
We also let $T = \Th_L(F)$.

\begin{definition}\label{defi:definable}
  Let $A \subseteq F$.
  We write $\Def(F/A)$
  for the Boolean algebra of $A$-definable subset in $F$,
  its operations being the set-theoretic ones.
  We also think of $\Def(F/A)$ as a BAM
  whose monotone operation $\mord$ is defined
  by
  \[
    \mord(\phi(F)) = (\mord\phi)(F)
  \]
  for an $L(A)$-formula $\phi(x)$, where $(\mord \phi)(x)$ is the $L$-formula
  $x \mop_y \phi(y)$.
  (It is easy to see that $\mord: \Def(F/A) \to \Def(F/A)$
    is well defined here.
  This is true of similar definitions that appear in later parts of the article.)
\end{definition}
It is easy to see that $\Def(F/A)$ is a subalgebra of $F^+$
as a BAM.

\begin{proposition}
  Assume $F' \models T$.
  Then $\Def(F/\0)$ and $\Def(F'/\0)$ are isomorphic as BAMs.
\end{proposition}

 Types play an important r\^ole in the proof of the original theorem of Fine
 as well as in this article.
\begin{definition}
    \hfill
  \begin{enumerate}[(i)]
  \item
    The \emph{Stone space $S_1(T)$ of $1$-types over $\0$ for $T$} is
    the ultrafilter frame $\uf(\Def(F/\0))$ of $\Def(F/\0)$.
    (Note that if $F'$ is such that $\Th_L(F') = T$, 
    then $\Uf(\Def(F'/\0)) = \Uf(\Def(F/\0))$ 
    and that $S_1(T)$ is, therefore,
    defined uniquely regardless of the choice of
    $F \models T$.)
    We consider $S_1(T)$ as a topological space whose open subsets
    are exactly the open elements
    of $(\Uf(\Def(F/\0)))^+ = (\Def(F/\0))^\sigma$.
    An element $p \in S_1(T)$ is called a \emph{$1$-type over $\0$}.
  \item
    Likewise, we let $S_1^F(A) = \uf (\Def(F/A))$.
    An element $p \in S_1^F(A)$ is called a \emph{$1$-type over $A$}.
  \item
    A set $\Sigma(x)$ of $L(A)$-formulas with one variable, say, $x$,
    is called a \emph{partial $1$-type over $A$}.
    We write $\Sigma(F)$ for the set
    $\{w \in F \mid \zenbu{\phi \in \Sigma} F \models \phi(w)\}$.
  \end{enumerate}
\end{definition}

\begin{convention}
  We identify a 1-type $p$ over $A$ with the partial 1-type
  \[\{\phi(x;\bar a) \mid \phi(F;\bar a) \in p, \bar a \in A\}\] over $A$.
  In fact, this is closer to how types are usually defined
  in classical model theory
  and is what types are in \cite{2017arXiv170103773L}.
  Likewise, we write $[\phi]$ for the clopen set $[X]$
  in a Stone space of 1-types if $\phi$ defines $X$.
\end{convention}

Given a partial type $\Sigma(x)$,
the intersection $\bigcap_{\phi \in \Sigma} [\phi]$
is a closed set in the Stone space of 1-types.

\begin{definition}
  \hfill
  \begin{enumerate}[(i)]
  \item 
    A partial 1-type $\Sigma(x)$ over $A$ is \emph{deductively closed}
    if $[\phi] \supseteq  \bigcap_{\psi \in \Sigma} [\psi]
    \implies \phi \in \Sigma$.
  \item
    For a deductively closed partial 1-type $\Sigma(x)$,
    we write $\closed{\Sigma}$
    for the closed set
    \[\{p \mid p \supseteq \Sigma\} = \bigcap_{\phi \in \Sigma} [\phi].\]
  \end{enumerate}
\end{definition}

\begin{proposition}
  Let $w \in F$ and $A \subseteq F$.
  The family $\tp^F(w/A)$ of A-definable subsets of $F$ containing $w$
  is an ultrafilter in $\Def(F/A)$
  and thus a 1-type over $A$.\qed
\end{proposition}

\begin{definition}\label{def:types2}
  \hfill
  \begin{enumerate}[(i)]
  \item
    Let $A \subseteq F$.
    An element $w\in F$ \emph{realizes} $p \in S_1^F(A)$,
    or $w \models p$,
    if $\tp^F(w/A) = p$.
    The 1-type $p$ is \emph{realized in} $F$
    if there is $w \in F$ with $w \models p$.
  \item\label{item:saturation}
    The $L$-structure $F$ is \emph{$\aleph_0$-saturated}
    if for every finite $A \subseteq F$,
    every $p \in S_1^F(A)$ is realized in $F$.
  \end{enumerate}
\end{definition}

\subsection{Model theory specific to neighborhood frames}
In this section, we study the model theory of neighborhood frames
while we relate it to the classical model theory.
\begin{definition}
  Let $L$ be a language of coalgebraic predicate logic based on $L_0$ and
  $F$ an $L$-structure.
  The \emph{essential part} $F^\ess$ of $F$
  is the $L$-structure whose reduct to $L_0$ is the same as that of $F$
  and whose neighborhood function $N^\ess$ is defined by
  \[
    U \in N^\ess(w) \iff \text{$U$ is definable in $F$ and $U \in N^F(w)$}
  \]
  for $w \in F^\ess$.
\end{definition}

\begin{proposition}[\cite{Chang1973-CHAMMT}]\label{prop:essential}
  Let $L$ be a language of coalgebraic predicate logic
  and $F, G$ be $L$-structures.
  Suppose $F^\ess \cong G^\ess$.
  \begin{enumerate}[(i)]
  \item $F \equiv_L G$.
  \item If $F$ is $\aleph_0$-saturated, so is $G$.\qed
  \end{enumerate}
\end{proposition}

We define a class of languages of first-order logic,
one for each language of coalgebraic predicate logic.
\begin{definition}[\cite{lmcs:1167},{\cite[Definition 9]{TENCATE2009146}}]\label{defi:twosorted}
  Let $L$ be an arbitrary language of coalgebraic predicate logic
  and $L_0$ the language of first-order logic on which $L$ is based.
  We define the language $L^2$ to be the two-sorted first-order language
  whose sorts are the \emph{state sort} and \emph{neighborhood sort}
  and whose atomic formulas are those in $L_0$,
  recast as formulas in which constants and variables belong to the state sort,
  together with $xNU$ and $x \in U$,
  where $x$ and $U$ are variables
  for the state sort and the neighborhood sort, respectively.
  (In general, we will use lowercase variables for the state sort
  and uppercase variables for the neighborhood sort.)
\end{definition}

\begin{definition}
  Let $L$ be a language of coalgebraic predicate logic
  and $F$  an $L$-structure.
  Given a family $\mathcal S \subseteq \power (F)$
  that contains all definable subsets of $F$,
  we write $(F,\mathcal S)$ for the $L^2$-structure $G$.
  The domain of the state sort of $G$ is that of $F$,
  and the domain of the neighborhood sort of $G$ is $\mathcal S$.
  The $L^2$-structure $G$
  interprets all nonlogical symbols of $L^2$ but $N$ and  $\in$
  in the same way as $F$.
  Finally, we have $(w, U) \in N^{G} \iff U \in N^F(w)$
  and $(w, U) \in \mathord{\in^{G}} \iff w \in U$.
  A family $\mathcal S$ is \emph{large for $F$}
  if $U \in \mathcal S$ whenever there is $w \in F$ with $U \in N^F(w)$.
%  We write $(F, \mathcal S)$ for $G$
%  and sometimes $F$ for $(F, \power(F))$.
\end{definition}

\begin{proposition}[\cite{Chang1973-CHAMMT,2017arXiv170103773L}]\label{prop:translation}
  Let $L$ be a language of coalgebraic predicate logic.
  Let $(-)^2: L \to L^2$ be the translation
  that commutes with Boolean combinations
  and satisfies
  \begin{align*}
    (\aru x\phi)^2 &= \aru x (\phi^2) \\
    (x\mop_y\phi)^2 &= \aru U [\zenbu y (y \in U \leftrightarrow \phi^2(y))
                               \land x N U].
  \end{align*}
  Let $\mathcal S \subseteq \power(F)$ be a family that
  contains all definable subsets of $F$.
  Then for every $L$-formula $\phi$ and $\bar a \in F$ we have
  \[
    F \models \phi(\bar a) \iff (F, \mathcal S) \models \phi^2(\bar a).\qedhere
  \]
\end{proposition}

\begin{remark}\label{rem:hoge}
  Note that the same two-sorted language $L^2$ is considered in
  \cite{lmcs:1167} even though their transformation of neighborhood
  frames into $L^2$-structures there is different from ours.
  While in \cite{lmcs:1167} a neighborhood frame $F$ is
  always associated with the structure $M$ for $L^2$ whose
  neighborhood sort consists of those subsets of $F$ that are
  neighborhoods of some state of $F$, we do not impose such a
  restriction here.  In addition, there is a third language for
  neighborhood frames used before as a model correspondence
  language~\cite[Definition 12]{TENCATE2009146} for neighborhood and
  topological semantics of modal logic and for the study of model
  theory of topological spaces~%
  \cite{flum80:_topol_model_theor} in general.  This is also a
  fragment of the two-sorted language introduced above and, in fact,
  contains the image of the embedding of coalgebraic predicate logic
  into the two-sorted language \cite{ziegler85:_chapt_xv}.
\end{remark}

\begin{lemma}\label{lem:familyofsets}
  Let $L$ be a language of coalgebraic predicate logic and
  $F$  an $L$-structure.
  Let $G$ be an $L^2$-structure that is an elementary extension of
  $(F, \power(F))$.
  There exists an $L$-structure $G'$
  whose domain is that of the state sort of $G$
  and a family $\mathcal S \subseteq \power(G')$
  that satisfies the following:
  \begin{enumerate}[(i)]
  \item $\mathcal S$ contains all definable subsets in $G'$.
  \item $\mathcal S$ is large for $G'$.
  \item $G \cong (G', \mathcal S)$. \label{item:familyofsets}
  \end{enumerate}
\end{lemma}
\begin{proof}
  Note that $F$ satisfies extensionality:
  \[
    (F, \power(F))
    \models \zenbu U \zenbu V [\zenbu x (x \in U \leftrightarrow x \in V)
    \rightarrow U = V].
  \]
  By $L^2$-elementarity, so does $G$.
  Let $G'$, $S^G$ be the domains of the state sort and
  the neighborhood sort of $G$, respectively.
  Let $i : S^G \to \power(G')$ be defined by
  \[i(U) = \{w \in G' \mid G \models w \in U\}.\]
  By the extensionality of $G$, $i$ is injective.
  Let $\mathcal S$ be the range of $i$.
  Define the neighborhood function $N^{G'}$ by
  \[ i(U) \in N^{G'}(w) \iff G \models w N U.\]

  Let $\phi(x; \bar y)$ be an $L$-formula
  and $X := \phi(G', \bar a)$ be a definable set in $G'$,
  where $\bar a \in G'$.
  Note that the $L^2$-structure $(F, \power(F))$ satisfies comprehension:
  \[
    (F, \power(F)) \models
    \zenbu {\bar y} \aru U \zenbu x
    (\phi^2(x; \bar y) \leftrightarrow x \in U).
  \]
  So does $G$. Let $U$ witness the satisfaction by $G$
  of the existential formula
  $\aru U \zenbu x(\phi^2(x; \bar a) \leftrightarrow x \in U)$.
  It can easily be seen that $i(U) = \phi(G', \bar a)$.
  
  It is easy to see that $\mathcal S$ is large for $G'$
  and that $G \cong (G', \mathcal S)$.
\end{proof}

\begin{proposition}\label{prop:saturation}
  Let $L$ be a language of coalgebraic predicate logic
  and $F$ an $L$-structure.
  There exists an $L$-structure $G$
  such that $G \equiv_L F$ and that $G$ is $\aleph_0$-saturated.%
  \footnote{Our use of both coalgebraic predicate logic
  and first-order logic makes phrases
  such as ``elementarily equivalent''
  and ``$\aleph_0$-saturation''
  potentially ambiguous because
  we have two different classes of definitions,
  one from the previous subsection
  and the other standard in classical model theory.
  Note, however, that
  (expansions of) neighborhood frames are never structures
  of any language of first-order logic
  and that first-order structures are never $L'$-structures
  for any language $L'$ of
  coalgebraic predicate logic.
  Hence, for example, if $L'$ is a language of coalgebraic predicate logic,
  and $F$ is an $L'$-structure,
  then whenever we say that $F$ is $\aleph_0$-saturated,
  we mean what we stated in Definition~\ref{def:types2}(\ref{item:saturation}),
  with $L$ in the definition being $L'$. }
\end{proposition}
\begin{proof}
  Consider the $L^2$-structure $(F, \power(F))$,
  and take an elementary extension
  $G_0$ of $(F, \power(F))$
  that is $\aleph_0$-saturated.
  By Lemma~\ref{lem:familyofsets}(\ref{item:familyofsets}),
  take an $L$-structure $G$ and $\mathcal S \subseteq \power(G)$
  with $G_0 \cong (G, \mathcal S)$.
  Suppose that $A \subseteq G$ is finite.
  Let $p \in S_1^{G}(A)$ be arbitrary.
%  We show that $p$ is realized in $G$.
  Let $\Sigma^2$ be the partial type
  $\{ \phi^2 \mid \phi \in p  \} $
  over $A$ in $L^2$.
  % There is a surjective continuous map $\pi$ from the Stone space in $L^2$
  % onto that in $L$.   Let $p^2 \in \pi^{-1}(p)$.
  Since $p$ is a proper filter in $\Def(F/A)$,
  the type $\Sigma^2$ is consistent by Proposition~\ref{prop:translation}.
  Thus, by the $\aleph_0$-saturation of $G_0$, we can take $w \in G_0$
  realizing $\Sigma^2$.
  By Proposition~\ref{prop:translation}, we have $w \models p$.
\end{proof}

We now introduce the notion of quasi-ultraproducts
as we will use it to give a proof 
of Fine's theorem at the end of this article.

\begin{definition}[\cite{Chang1973-CHAMMT,2017arXiv170103773L}]
  Let $L$ be a language of coalgebraic predicate logic based on $L_0$
  and $(F_i)_{i \in I}$ be a family of monotonic $L$-structures.
  Suppose that $D$ is an ultrafilter over $I$.
  Let $\prod_D F_i$ be the ultraproduct of $(F_i)_i$ as $L_0$-structures
  modulo $D$.
  A subset $A \subseteq \prod_D F_i$ is \emph{induced by}
  a family $(A_i)_{i \in J}$ 
  if $J \in D$, $A_i \subseteq F_i$ for $i \in J$,
  and
  \[
    a \in A \iff a(i) \in A_i \text{ for all $i \in J$.}
  \]
  A \emph{quasi-ultraproduct of $(F_i)_i$ modulo $D$}
  is a monotonic $L$-structure
  that is
  the $L_0$-structure $\prod_D F_i$ equipped with a neighborhood function
  $N$
  that satisfies
  \[
    A \in N(w) \iff A_i \in N^i(w(i))
    \text{ for all $i \in J$},
  \]
  whenever $w \in \prod_D F_i$, and $A$ is induced by $(A_i)_{i \in J}$.
  \sloppy
  A class $\mathcal K$ of monotonic neighborhood frames
  \emph{admits quasi-ultraproducts}
  if whenever $(F_i)_i$ is a family of neighborhood frames from $\mathcal K$,
  a quasi-ultraproduct of $(F_i)_i$ exists in $\mathcal K$.
\end{definition}
\fussy
\begin{proposition}[\cite{2017arXiv170103773L,Chang1973-CHAMMT}]\label{prop:qup}
  \hfill
  \begin{enumerate}
  \item
    Each class of the classes in Table~\ref{tab:defs}
    admits quasi-ultraproducts.
  \item
    Let $(F_i)_{i\in I}$ be a family of monotonic $L$-structures
    for a language $L$ of coalgebraic predicate logic.
    If $F_i$ satisfies a theory $T$ for all $i \in I$,
    so does a quasi-ultraproduct of $(F_i)_i$.
  \end{enumerate}
\end{proposition}
\begin{proof}
  \hfill
  \begin{enumerate}
  \item
    By Remark~\ref{rem:elementary}, 
    it suffices to prove this for the class of monotonic neighborhood frames,
    the class of quasi-filter frames.
    This could be done by using the machinery introduced
    in Litak et al.~\cite{2017arXiv170103773L},
    but it is easy to prove it directly in the following  way.

    Let $\mathcal K_0$ be either the class of monotonic neighborhood frames
    or the class of quasi-filter neighborhood frames.
    Let $(F_i)_i$ be a family of neighborhood frames in $\mathcal K_0$.
    Let $N^i$ be the neighborhood function of $F_i$.
    Define the neighborhood function $N$ on $\prod_D F_i$ as follows:
    a subset $U \subseteq \prod_D F_i$ is in $N(w)$ if and only if
    there is a set $A \subseteq U$ induced by
    $(A_i)_{i \in J}$
    with $A_i \in N^i(w(i))$ for all $i \in J$.
    It is easy to see that this indeed defines a quasi-ultraproduct
    and that if each $F_i$ is in $\mathcal K_0$
    then so is the quasi-ultraproduct.

  \item
    The usual argument by induction works;
    see Litak et al.~\cite{2017arXiv170103773L}.\qedhere
  \end{enumerate}
\end{proof}

\section{Proof of the main lemmas}
\label{sec:canonicity}
In this section we prove the main lemmas of this article.
Recall that $L_=$ is the language of coalgebraic predicate logic
based on the empty language of first-order logic.

\begin{definition}\label{defi:CPL-elementary}
  Let $\mathcal K_0$ be a class of monotonic neighborhood frames.
  A class $\mathcal K$ of monotonic neighborhood frames is
  \emph{CPL-elementary relative to $\mathcal K_0$}
  if there is an $L_=$-theory $T$ with
  \[\mathcal K = \{ F \in \mathcal K_0 \mid F \models T \}.\]
  Two monotonic neighborhood frames $F$ and $F'$ are
  \emph{CPL-elementarily equivalent relative to $\mathcal K_0$}
  if $F, F' \in \mathcal K_0$ and $\Th_{L_=}(F) = \Th_{L_=}(F')$.
\end{definition}

\begin{remark}\label{rem:elementary}
  The class of filter frames is CPL-elementary
  relative to the class of quasi-filter frames
  (see Definition~\ref{defi:CPL-elementary}),
  and the class of augmented filter frames is CPL-elementary
  relative to the class of augmented quasi-filter frames;
  indeed, they are both defined by the same $L_=$-sentence 
  $\zenbu{x} x \mop_y y = y$.
  Furthermore,
  by the second paragraph of Example~\ref{exa:T},
  the class of augmented quasi-filter frames is CPL-elementary
  relative to the class of monotonic frames.
  Therefore,
  the main lemma in this section concerns the classes of monotonic
  and quasi-filter neighborhood frames, respectively,
  which suffice for the purpose of the main results
  (Theorems~\ref{thm:gtt} and \ref{thm:main}),
  which deal with any of the classes in Table~\ref{tab:defs}.
\end{remark}

\begin{lemma}\label{lem:dontwannamove}
  Let $F$ be a monotonic neighborhood frame,
  and let $G$ and $G'$ be $(L_=)^2$- and $L_=$- structures,
  respectively,
  obtained 
  by elementarily extending $F$ as in Lemma~\ref{lem:familyofsets}.
  \begin{enumerate}[(i)]
  \item \label{item:definablemonotonicity}
    If $F$ is monotonic, $X, Y \subseteq G'$ are definable, $X \subseteq Y$,
    and $X \in N^{G'}(w)$ for $w \in G'$, then $Y \in N^{G'}(w)$.
  \item \label{item:definableregularity}
    If $F$ is an augmented filter frame, then for every $w \in G'$
    either $N^{G'}(w)$ is empty
    or  has a minimum element.
  \end{enumerate}
\end{lemma}
\begin{proof}
  For (\ref{item:definablemonotonicity}),
  let $L(G')$-formulas $\phi(x; \bar a)$ and $\psi(x; \bar b)$
  define $X$ and $Y$,
  respectively.
  Since $F$ is monotonic, we have
  \begin{align}
    \label{eq:thedisplayedformula}
    (F, \power(F)) \models
    \zenbu{\bar y} \forall{\bar z} \zenbu{v}
    [&\zenbu x (\phi(x; \bar y) \rightarrow \psi(x; \bar z)) \notag\\
     &\land v\mop_x\phi(x; \bar y)
     \rightarrow v\mop_x\psi(x; \bar z)].
  \end{align}
  Since $(F,\power(F))$ e satisfies the $(-)^2$-translation
  of the right-hand side of the  displayed formula~(\ref{eq:thedisplayedformula})
  by Proposition~\ref{prop:translation},
  so does $G$.
  Again by Proposition~\ref{prop:translation},
  \[
    G' \models
    \zenbu x (\phi(x; \bar a) \rightarrow \psi(x; \bar b))
     \land w\mop_x\phi(x; \bar a)
     \rightarrow w\mop_x\psi(x; \bar b).
  \]
  Since $X \in N^{G'}(w)$, we have $\psi(G', \bar b) \in N^{G'}(w)$.

  For (\ref{item:definableregularity}),
  first observe that the $L^2$-structure $(F, \power(F))$ satisfies the sentence
  \[
    \zenbu{x}[\neg\aru{U} xNU \vee
    \aru{U_0}\zenbu{U}(xNU \rightarrow U_0 \subseteq U)],
  \]
  where $\subseteq$ is an abbreviation of the obvious $L^2$-formula.
  Since $G'$ satisfies the same $L^2$-formula,
  the claim follows.
\end{proof}

We are now ready to prove the key lemmas used in the proof of our main result.
Our lemmas are analogous to \cite[8.9 Theorem]{vanBenthem1983-VANMLA}.

\begin{lemma}\label{lem:main}
  Let $F$ be a monotonic neighborhood frame.
  There exists $G \equiv_{L_=} F$
  such that there is a surjective bounded morphism
  $f : G \twoheadrightarrow \ue F$.
  Moreover,
  if $\mathcal K_0$ is either
  the class of monotonic neighborhood frames or
  the class of quasi-filter neighborhood frames,
  and
  $F \in \mathcal K_0$, then we can take $G \in \mathcal K_0$.
\end{lemma}
The following is the outline of the proof, which comes after this paragraph.
We follow the classical proof of \cite[8.9 Theorem]{vanBenthem1983-VANMLA}
by taking an expansion $L$ of the correspondence language
so that every subset of the given frame $F$ will be definable
and taking an $\aleph_0$-saturated extension $G$ in that language.
However, we need to add more neighborhoods
to the neighborhood frame $G$ that is being constructed
to make sure that the map from $G$ to the ultrafilter frame of $F$
is a bounded morphism.
Much of the proof is dedicated
to showing that this construction preserves elementary equivalence in $L$.
\begin{proof}
  Let $L$ be the language of coalgebraic predicate logic
  based on $\{P_S \mid S \subseteq F\}$,
  the unary predicates for the subsets of $F$.
  The neighborhood frame $F$ can be made into an $L$-structure naturally.
  Let $G_0 \equiv_L F$ be an $\aleph_0$-saturated $L$-structure
  as obtained by Proposition~\ref{prop:saturation}.
  Let $G_1$ be the essential part of $G_0$.
  Let $G_2$ be the $L$-structure obtained from $G_1$ as follows:
  for each state $w \in G_1$,
  add as  a neighborhood of $w$
  the set $\Sigma(G_1)$,
  where $\Sigma(x)$ is a partial type over a finite set $A \subseteq G_1$
  such that $\Sigma(x)$ is deductively closed
  % modulo $T = \Th_L(F)$
  % the full $L$-theory of $G_1$ (and of $F$),
  and that for every $\phi \in \Sigma$
  we have $\phi(G_1) \in N^{G_1}(w)$.
  We call such a partial type \emph{good at $w$}.
  Let $G$ be the $L$-structure identical to $G_2$
  except that its neighborhood function $N^G$ is defined by
  $U \in N^G(w) \iff \aru{U_0 \subseteq U} U_0 \in N^{G_2}(w)$.

  Note that a singleton partial type $\Sigma = \{\phi\}$ with $\phi(x) \in L(A)$
  is always good at $w \in G_1$
  if $\phi(G_1) \in N^{G_1}(w)$.
  
  We show that $G \equiv_L F$.
  By Proposition~\ref{prop:essential},
  we have $G_1 \equiv_L G_0 \equiv_L F$,
  so
  it suffices to see that for every definable $X \subseteq G$
  we have $X \in N^G(w) \iff X \in N^{G_1}(w)$.
  We show $\implies$ (the other direction is easy).
  By construction,
  there is either a definable set $Y \subseteq X$ with $Y \in N^{G_1}(w)$
  or a partial type $\Sigma(x)$ over a finite set $A$
  good at $w$ with $\Sigma(G_1) \subseteq X$.
  The former is a special case of the latter, so we assume the latter.
  Let $A'$ be a finite set containing $A$ and the parameters used in the definition of $X$. 
  Let $f' : G_1 \twoheadrightarrow S_1^{G_1}(A')$ be defined by
  $f'(w) = \tp^{G_1}(w/A')$.
  By $\aleph_0$-saturation,
  $f'$ is a surjection.
  We show that $f'(\Sigma(G_1)) = \closed{\Sigma} \subseteq S_1^{G_1}(A')$.
  It is easy to show that $f'(\Sigma(G_1)) \subseteq \closed{\Sigma}$;
  we show $f'(\Sigma(G_1)) \supseteq \closed{\Sigma}$.
  Let $p \in \closed{\Sigma}$ be arbitrary.
  By $\aleph_0$-saturation, take $w \in G_1$ with $f'(w) = p$.
  Since $p \supseteq \Sigma$, $w \in \Sigma(G_1)$.
  We have shown that $f'(\Sigma(G_1)) = E_\Sigma$.
  That $f'(X) = [X]$  easily follows from the $\aleph_0$-saturation of $G$ as well.
  We have $\closed{\Sigma} \subseteq [X]$.
  By the compactness of $S_1^{G_1}(A')$,
  we have a finite $\Sigma_0 \subseteq \Sigma$
  for which $\closed{\Sigma_0} \subseteq [X]$.
  Being the intersection of finitely many clopen sets,
  \[\closed{\Sigma_0} = \bigcap_{\phi \in \Sigma_0}[\phi]
    = \left[ \bigwedge \Sigma_0 \right] \]
  is clopen.
  Since $\Sigma$ is good at $w$, we have
  $(\bigwedge \Sigma_0) (G_1) \in N^{G_1}(w)$.
  We conclude that $X \in N^{G_1}(w)$
  by Lemma~\ref{lem:familyofsets}~(\ref{item:definablemonotonicity}).
  (See Remark~\ref{rem:excuse} for an alternate proof of this fact.)

  Since $F^+ \cong \Def(F/\0)$, we have $\ue F \cong S_1(T)$,
  where $T$ is the full $L$-theory of $F$,
  which is identical to $\Th_L(G)$.
  We show $f : G \twoheadrightarrow S_1(T)$ defined
  by $f(w) = \tp^G(w/\0)$,
  which is surjective by $\aleph_0$-saturation,
  is a bounded morphism.
  In the rest of the proof, we write $N^\sigma$
  for the neighborhood function of $S_1(T)$.

  \paragraph*{The ``forth'' condition.}
  Suppose that $U \in N^G(w)$.
  We show that $f(U) \in N^\sigma(\tp^G(w))$.
  % We write $U$ for $f^{-1}(U')$.
  By construction,
  we have either (I) $U\supseteq \phi(G, \bar a) \in N^G(w)$
  or (II) $U \supseteq \Sigma(G) \in N^G(w)$,
  where $\phi(x,\bar y)$ is an $L$-formula,
  $\bar a \in G$,
  and $\Sigma(x)$ is a partial type over a finite set $A$
  good at $w$.
  Since (I) is a special case of (II), we will just show (II).

  % For (I), assume that $U\supseteq \phi(G, \bar a) \in N^G(w)$.
  % Let
  % \[K = \{q \in S_1(T)
  %   \mid \mathop{(\exists q' \in S_1^G(\bar a))}
  %   \phi(x,\bar a) \in q'\}.\]
  % Being the image of a clopen set
  % under the restriction map $S_1^G(\bar a) \twoheadrightarrow S_1(T)$,
  % which is continuous and thus closed,
  % $K$ is a closed set.
  % Note that $\mord\chi \in f(w) = \tp^G(w/\0)$ if and
  % only if $\chi(G) \in N^G(w)$.
  % For (i),
  % assume that $[\chi] \supseteq K$.
  % For arbitrary $v \in \phi(G, \bar a)$,
  % since $\tp^G(v/\bar a')$ contains $\phi(x, \bar a)$,
  % $\tp^G(v/\0)$ is in $K$ and thus contains $\chi(x)$.
  % Hence, $\chi(G) \supseteq \phi(G, \bar a)$,
  % and by monotonicity $\chi(G) \in N^G(w)$.
  % For (ii),
  % let $q \in K$ be arbitrary.
  % It suffices to show that $q$ is realized by some $v \in U$.
  % By the definition of $K$,
  % there exists $q' \in S_1^G(\bar a)$ that extends $q$
  % with $\phi(x, \bar a) \in q'$.
  % By $\aleph_0$-saturation,
  % there is some $v \in G$ that realizes $q'$;
  % since $\phi(x, \bar a) \in q'$, we have $v \in \phi(G, \bar a) \subseteq U$.

  For (II),
  assume that $U \supseteq \Sigma(G) \in N^G(w)$,
  where $\Sigma$ is a partial 1-type over finite $A$ good at $w$.
  Let $K = r(\closed{\Sigma})$,
  where $r : S_1^G(A) \to S_1(T)$ is the closed continuous map
  dual to the embedding $\Def(G/\0) \hookrightarrow \Def(G/A)$.
  Note that $r(q) = q \cap \Def(G/\0)$ for $q \in S_1^G(A)$.
  Being the image of a closed map of a closed set, $K$ is closed.
  Recall the equation~(\ref{eq:ultraframe}) that defines $N^\sigma$
  to see that
  it suffices to show (i) that for every $\chi(x) \in L$
  we have $[\chi] \supseteq K \implies \chi(G) \in N^G(w)$
  and (ii) that $K \subseteq f(U)$.
  For (i), assume that $[\chi] \supseteq r(\closed{\Sigma})$,
  where $\chi(x) \in L$, and $[\chi]$ denotes a subset in $S_1(T)$.
  Take an arbitrary $q \in \closed\Sigma$.
  Then $r(q) \in r(\closed\Sigma) \subseteq [\chi]$,
  so $\chi \in r(q) \subseteq q$.
  We have just shown that $[\chi] \supseteq \closed{\Sigma}$,
  where $[\chi]$ denotes a subset in $S_1^G(A)$.
  By deductive closure $\chi \in \Sigma$.
  By construction, $\chi(G) \in N^G(w)$.
  For (ii),
  it suffices to show that arbitrary $q \in \closed{\Sigma}$ can be realized
  by an element of $U$.
  Since $q$ is a type over a finite set, by $\aleph_0$-saturation, we may take $v \models q$;
  this means $v \models \Sigma$, i.e., $v \in \Sigma(G) \subseteq U$.

  \paragraph*{The ``back'' condition.}
  Suppose that $U' \subseteq S_1(T)$ is in $N^\sigma(\tp^G(w/\0))$.
  We show that there is $U \subseteq G$ in $N^G(w)$
  such that $f(U) \subseteq U'$.
  By the definition of $N^\sigma$,
  there is a partial type $\Sigma(x)$ over  $\0$ good at $w$
  such that $\closed{\Sigma} \subseteq U'$.
  By construction, $\Sigma(G) \in N^G(w)$.
  Let $U: = \Sigma(G)$.
  Then for every $v \in U$, the type $\tp^G(v/\0)$ extends $\Sigma$
  and thus is in $\closed{\Sigma} \subseteq U'$.

  \paragraph*{Closure in relatively CPL-elementary classes.}
  By construction, $G$ is monotonic.

  Suppose that $F$ is a quasi-filter neighborhood frame.
  Let $w \in G$ and $U, U' \in N^G(w)$ be arbitrary.
  By construction, there are deductively closed partial types
  $\Sigma(x), \Sigma'(x)$
  over a finite set of parameters both of which are good at $w$
  with $\Sigma(G) \subseteq U$ and $\Sigma'(G) \subseteq U'$.
  The partial type $\Sigma \cup \Sigma'$ is also over a finite set,
  good at $w$.
  Moreover, $\Sigma \cup \Sigma'$ is deductively closed
  since $F$ is a quasi-filter frame.
  Therefore, we have
  $(\Sigma \cup \Sigma')(G) = \Sigma(G) \cap \Sigma(G) \subseteq U \cap U'$,
  so $U \cap U' \in N^G(w)$.
  We have seen that $G$ is a quasi-filter neighborhood frame.
\end{proof}

\begin{remark}\label{rem:excuse}
  In the proof above,
  we obtain $G$ not only by compactness
  but also by altering the neighborhoods in an ad-hoc way
  while maintaining elementary equivalence in $L_=$.
  There is no reason for us to believe that
  $G$ has the same theory
  as $F$ in ${L_=}^2$ or in the languages described in Remark~\ref{rem:hoge}.
  This is why we find it difficult to extend our main result
  to the more expressive languages.

  The following is the alternate proof that I announced
  at the end of the third paragraph of the proof
  (the concepts that we have not defined have obvious definitions):
  Suppose $X$ is definable by $\psi(x;A')$ where $\psi \in L$
  and $A' \subseteq G$ is a finite set.
  By $\aleph_0$-saturation of $G_1$,
  we have $\Th_{L(A')} (G_1) \cup \Sigma(x) \models \psi(x, A')$
  (otherwise, realize the type $\Sigma(x) \cup \{\neg\psi(x, A')\}$
  by some element in $G_1$, which would be in $\Sigma(G_1) \setminus X$.)
  By compactness, there is finite $\Sigma_0 \subseteq \Sigma$
  such that $\Sigma_0(G_1) \subseteq \psi(G_1, A')$.
  Since $\bigwedge \Sigma_0(x)$ is a single formula of $L$,
  by deductive closure $\bigwedge \Sigma_0(x) \in \Sigma(x)$.
  Hence $\bigwedge \Sigma_0(G_1) \in N^{G_1}(w)$.
  By Lemma~\ref{lem:familyofsets}(\ref{item:definablemonotonicity}),
  we have $X = \psi(G_1, A') \in N^{G_1}(w)$ as desired.
\end{remark}

\section{Applications of the main lemmas}\label{sec:appl-main-lemm}
\subsection{The Goldblatt-Thomason Theorem}
An algebraic argument essentially the same as the classical counterpart
can be used to show that a class of monotonic neighborhood frames 
closed under ultrafilter extensions is modally definable
if and only if it is closed under bounded morphic images,
generated subframes, and disjoint unions, 
and it reflects ultrafilter extensions
\cite{kurz2007goldblatt}\cite[Theorem 7.23]{hansen03:_monot_modal_logic}.
% (moreover, if such a class is modally definable,
% then it is definable by a canonical set of formulas).
By applying Lemma~\ref{lem:main},
we obtain the following theorem.
\begin{theorem}
  \label{thm:gtt}
  Let $\mathcal K$ be a class of monotonic neighborhood frames
  that is closed under CPL-elementary equivalence
  relative to any of the classes in Table~\ref{tab:defs}.
  $\mathcal K$ is modally definable
  if and only if it is closed under bounded morphic images,
  generated subframes, and disjoint unions, 
  and it reflects ultrafilter extensions.
  % Moreover, if such a class is modally definable,
  % then it is definable by a canonical set of  formulas.
\end{theorem}
\begin{proof}
  Let $\mathcal K$ be a class of monotonic neighborhood frames
  that is closed under CPL-elementary equivalence
  relative to a class $\mathcal K_0$ in Table~\ref{tab:defs}.
  We show the ``if'' case.
  Suppose that $\mathcal K$ is closed under bounded morphic images,
  generated subframes, and disjoint unions, 
  and it reflects ultrafilter extensions.
  Apply Lemma~\ref{lem:main} and Remark~\ref{rem:elementary}
  to conclude that $\mathcal K$ is closed under ultrafilter extensions.
  Note that the hypothesis of \cite[Theorem 7.23]{hansen03:_monot_modal_logic}
  is satisfied,
  and we conclude that 
  $\mathcal K$ is modally definable.
\end{proof}

\begin{example}
  As an example,
  we show that the image $\mathcal K$ under ${}^*$ of 
  the class of discrete topological spaces is modally definable.
  For a quasi-filter frame $F$, $F$ is a ${}^*$-image 
  of a discrete topological space
  if and only if $F \models \zenbu{x} \neg x\mop_z y \neq x$ and $F \models \zenbu{x} x\mop_y y = x$.
  Hence,
  $\mathcal K$ is CPL-elementary relative to the class of quasi-filter frames,
  and
  the Goldblatt-Thomason Theorem is applicable to $\mathcal K$.
  It is easy to check that $\mathcal K$ is closed under bounded morphic images,
  generated subframes, and disjoint unions,
  so it suffices to show that $\mathcal K$ reflects ultrafilter extensions.
  Assume that for a neighborhood frame $F = (F, N)$
  its ultrafilter extension $\ue F = (\ue F, N^\sigma)$ is in $\mathcal K$.
  We show that $F \in \mathcal K$.
  The class of topological frames is defined
  by modal formulas $\mord p \land \mord q \to \mord(p \land q)$,
  $\mord p \to p$,
  and $\mord\mord p \to \mord p$ \cite{vanBenthem2007},
  so we may assume that $F$ is topological
  as ultrafilter extensions reflect modally definable classes.
  Let $w\in F$ be arbitrary, and
  let $u$ be the principal ultrafilter generated by $w$,
  so $u\in \ue F$.
  Note that $U \in N^\sigma(u) \iff u \in U$ 
  since $\ue F$ is the ${}^*$-image of a discrete space.
  Recall the definition of $N^\sigma$ in (\ref{eq:ultraframe}).
  The singleton $\{u\}$ is in $N^\sigma(u)$,
  and this has to be witnessed by $K = \0$ or $K = \{u\}$ 
  according to (\ref{eq:ultraframe}) of Definition~\ref{def:algebraic-concepts}.(\ref{item:please}).
  Suppose $K = \0$.
  Then (\ref{eq:ultraframe}) implies
  that $\mord^{F^+} \0 \in u$ among other things 
  (recall that $A$ in (\ref{eq:ultraframe}) is $F^+$ here).
  However, since $F$ is topological, $\mord^{F^+} \0 = \0$,
  and it cannot belong to an ultrafilter $u$.
  Hence, $K = \{u\}$.
  Again by (\ref{eq:ultraframe}),
  for all $a \subseteq F$ such that $[a] \supseteq K = \{u\}$,
  i.e., $a \in u$, we have that $\mord^{F^+} a \in u$.
  Let $a = \{w\}$,
  so $a \in u$.
  Since the set $u$ is an ultrafilter,
  we have $a \wedge \mord^{F^+} a \neq \0$,
  that is,
  $a \cap \{w \in F\mid a \in N(w)\} \neq \0$;
  this implies $\{w\} \in N(w)$.
  Since $w$ was arbitrary,
  we conclude that $F\in \mathcal K$.
  We have shown that $\mathcal K$ is modally definable;
  in fact, it is defined by $p \to \mord p$ 
  in addition to the definition of the class of
  topological neighborhood frames.
\end{example}

\subsection{Fine's Canonicity Theorem}
By the dual equivalence between monotonic modal logics
and varieties of BAMs \cite[Chapter 7]{hansen03:_monot_modal_logic},
we will state our version of Fine's Canonicity Theorem in an algebraic manner.
Our presentation of the proof of the theorem is modeled after 
that of the classical version of the theorem in \cite{venema07:_algeb_coalg}.

For a class $\mathcal K$ of neighborhood frames,
we write $\mathcal K^+$ for the class $\{F^+ \mid F \in \mathcal K\}$.

\begin{lemma}\label{lem:twoclosures}
  Let $\mathcal K$ be a class CPL-elementary relative to any of the classes
  in Table~\ref{tab:defs}.
  Let $\mathcal S \supseteq \mathcal K^+$ be the least class of BAMs
  closed under subalgebras.
  \begin{enumerate}
  \item
    $\mathcal S$ is closed under canonical extensions.
  \item
    $\mathcal S$ is closed under ultraproducts.
  \end{enumerate}
\end{lemma}
\begin{proof}
  \hfill
  \begin{enumerate}
  \item 
    Let $A \in \mathcal S$.
    For some $F \in \mathcal K$
    we have $A \hookrightarrow F^+$.
    By duality theory \cite[Theorem 5.4]{GEHRKE2001345}, we have
    $A^\sigma \hookrightarrow (F^+)^\sigma$.
    By Lemma~\ref{lem:main} and Remark~\ref{rem:elementary},
    there is $G \in \mathcal K$ with $(F^+)^\sigma \hookrightarrow G^+$.
    Thus, we have $A^\sigma \in \mathcal S$ by definition.

  \item 
    \newcommand{\pu}{\mathrm{pu}}
    \newcommand{\cm}{\mathrm{cm}}
    It suffices to do the following:
    given an ultraproduct $\prod_D F_i^+$
    where $I$ is an index set, $D$ is an ultrafilter over $I$,
    and $(F_i)_i$ is a family of neighborhood frames in $\mathcal K$,
    we show that the ultraproduct embeds into
    $\left( \prod_D F_i \right)^+$,
    where $\prod_D F_i$ is a quasi-ultraproduct of $(F_i)_i$ modulo $D$.
    In fact, we show that
    $\iota: \prod_D F_i^+ \to \left( \prod_D F_i \right)^+$
    defined by
    \[
      s \in \iota(a) \iff \{i \mid s(i) \in a(i) \} \in D,
    \]
    where $s \in \prod_D F_i$ and $a \in \prod_DF_i^+$
    is a BAM embedding
    (we do not write equivalence classes modulo $D$ explicitly;
    it is easy to see that $\iota$ is well defined).
    It can easily be seen that $\iota$ is a Boolean algebra embedding.
    We show that $\iota \circ \mord^\pu = \mord^\cm \circ \iota$,
    where $\mord^\pu$ and $\mord^\cm$ are the operations of
    the domain and the target of $\iota$, respectively.
    Let $N$ be the neighborhood function of the quasi-ultraproduct.
    We write $\mord^i$ and $N^i$ for the operation of $F_i^+$.
    Note that for all $a$ the set $\iota(a)$ is
    an induced subset of the quasi-ultraproduct;
    if we let $\pi_i(A)$ be the projection of an induced subset $A$
    of the quasi-ultraproduct onto the coordinate $i$,
    then  $\{i \mid \pi_i(\iota(a)) = a(i)\} \in D$.
    We now have
    \begin{align*}
      s \in (\iota \circ \mord^\pu)(a)
      &\iff \{i \mid s(i) \in (\mord^\pu(a))(i)\} \in D\\
      &\iff \{i \mid s(i) \in \mord^i(a(i))\} \in D \tag{*}\\
      &\iff \{i \mid s(i) \in \mord^i(\pi_i(\iota(a)))\} \in D \\
      &\iff \{i \mid \iota(a) \in N^i(s(i))\} \in D\\
      &\iff \iota(a) \in N(s)\\
      &\iff s \in (\mord^\cm \circ i)(a),
    \end{align*}
%    \sloppy
    where we have the equivalence (*)
    since \[\{i \mid (\mord^\pu(a))(i) = \mord^i(a(i))\} \in D.\qedhere\]
  \end{enumerate}
\end{proof}

\begin{theorem}\label{thm:main}
  Let  $\mathcal K$ be a class CPL-elementary relative to
  any of the classes in Table~\ref{tab:defs}.
  The variety of BAMs generated by $\mathcal K^+$ is canonical,
  \ie, closed under canonical extensions.
\end{theorem}
\begin{proof}
  Recall Remark~\ref{rem:elementary}.
  Gehrke and Harding~\cite{GEHRKE2001345} showed that
  if $\mathcal S$ is a class of BAMs closed under ultraproducts
  and canonical extensions, then $\mathcal S$ generates a canonical variety.
  Apply this result for the class $\mathcal S$ in Lemma~\ref{lem:twoclosures}
  to conclude that the variety generated by $\mathcal K^+$,
  which is identical to the variety generated by $\mathcal S$,
  is canonical.
\end{proof}

Note that
Fine's original theorem follows as a special case concerning the
classes of augmented neighborhood frames.

\begin{example}
  Consider the B axiom $p \rightarrow \mord \neg \mord \neg p$,
  which we considered in Example~\ref{exa:B}.
  Recall that it defined a CPL-elementary class $\mathcal K$
  relative to the class of monotonic neighborhood frames.
  % By the usual minimum valuation argument,
  % it is easy to see that 
  % this axiom defines a class $\mathcal K$ elementary relative
  % to the class of monotonic neighborhood frames;
  % this class is defined by the following sentence of $L_=$:
  % \[
  %   \zenbu x x \mop_y \neg y \mop_z z \neq x.
  % \]
  By \cite[Proposition 6.5]{hansen03:_monot_modal_logic},
  the variety $\mathcal V$ defined
  by the B axiom is canonical
  and hence generated by $\mathcal K^+$.
  By Theorem~\ref{thm:main},
  the canonicity of $\mathcal V$ is explained
  by the CPL-elementarity of $\mathcal K$.
\end{example}
\begin{remark}
  By Remark~\ref{rem:important},
  Theorem~\ref{thm:main} can be used to show the canonicity of the monotonic modal logic
  axiomatized by any formula of the form~(\ref{eq:veryfirst}).
\end{remark}

\section{Open questions}
As we mentioned in Remarks~\ref{rem:hoge} and \ref{rem:excuse},
one could attempt to use a different notion of elementarity
in stating and proving the results of this article,
but we stuck to coalgebraic predicate logic
due to the limitation of the proof technique we used.
A natural question to ask here would be whether
there is a more expressive first-order-like logic
that admits similar results
possibly by a different kind of proof.
Another question would be to characterize classes
of monotonic neighborhood frames
that admit analogues of the Goldblatt-Thomason theorem 
and Fine's theorem in the same sense
as in the main result of this article.
This question leads to another problem  of
showing results similar to ours
for other coalgebras than
those discussed in this article.

It was suggested to the author that
our version of Fine's theorem could be proved by
using an algebraic result~\cite{gehrke2006macneille},
which implies the original, Kripke-semantic version of the theorem.
The ``proof'' proposed contained a gap,
and therefore it remains open
whether the results in this article follow
from the aforementioned algebraic theorem.
Even if they can indeed be proved in that manner,
we hope that the proof presented here serves our original purpose
of investigating the role of coalgebraic predicate logic
in the study of monotonic modal logics,
especially in the spirit of van Benthem's program~%
\cite{benthem05:_univer_algeb_model_theor} of
re-analyzing algebraic arguments occurring in modal logic
from a model-theoretic perspective.

\paragraph{Acknowledgements.} 
I wish to give special thanks to Wesley Holliday for his extensive
and helpful comments and discussion.  I also wish to thank 
Tadeusz Litak, Lutz Schr\"oder, and Frederik Lauridsen
for useful comments on earlier drafts.
Finally, I gratefully acknowledge financial support
from the Takenaka Scholarship Foundation.  

\bibliographystyle{sl_style/sl.bst}
\bibliography{blah.bib}
% \printbibliography

%%%%%%  AUTHOR'S ADDRESS INFORMATION AT THE END OF THE PAPER:

\AuthorAdressEmail{Kentar\^o Yamamoto}%
{Group in Logic and the Methodology of Science\\
University of California, Berkeley\\
Berkeley, California USA}%
{ykentaro@math.berkeley.edu}

\end{document}